\documentclass[letterpaper,11pt]{amsart}
\usepackage{amssymb,amsmath, amsthm, enumerate, graphicx}
\usepackage{geometry}[1in]
\title{Pseudocompact C$^*$-algebras}
\author{Stephen Hardy}

\address{P.O. Box 142, Hampden-Sydney College, Hampden-Sydney, VA 23943}
\email{srhardy@hsc.edu}

\newcommand{\A}{\mathfrak{A}}
\renewcommand{\H}{\mathcal{H}}
\newcommand{\B}{\mathcal{B}}
\newcommand{\K}{\mathcal{K}}
\newcommand{\U}{\mathcal{U}}

\newcommand{\Z}{\mathcal{Z}}

\DeclareMathOperator{\tr}{tr}

\renewcommand{\u}{\mathcal{U}}
\DeclareMathOperator{\V}{\mathcal{V}}
\DeclareMathOperator{\dist}{dist}
\DeclareMathOperator{\conv}{conv}

\DeclareMathOperator{\diag}{diag}

\theoremstyle{definition}
\newtheorem*{thm*}{Theorem}
\newtheorem{thm}{Theorem}[section]
\newtheorem*{prop*}{Proposition}
\newtheorem{prop}[thm]{Proposition}
\newtheorem*{lemma*}{Lemma}
\newtheorem{lemma}[thm]{Lemma}
\newtheorem*{cor*}{Corollary}
\newtheorem{cor}[thm]{Corollary}
\newtheorem*{defn*}{Definition}

\newtheorem*{question*}{Question}

\newtheorem*{conj*}{Conjecture}
\newtheorem{conj}[thm]{Conjecture}
\newtheorem*{ex*}{Example}

\begin{document}

\begin{abstract}
We study the class of pseudocompact C$^*$-algebras, which are the logical limits of finite-dimensional C$^*$-algebras.  The pseudocompact C$^*$-algebras are unital, stably finite, real rank zero, stable rank one, and tracial.  We show that the pseudocompact C$^*$-algebras have trivial $ K_1 $ groups and the Dixmier property.  The class is stable under direct sums, tensoring by finite-dimensional C$^*$-algebras, taking corners, and taking centers.  We give an explicit axiomatization of the commutative pseudocompact C$^*$-algebras.  We also study the subclass of pseudomatricial C$^*$-algebras, which have unique tracial states, strict comparison of projections, and trivial centers.   We give some information about the $ K_0 $ groups of the pseudomatricial C$^*$-algebras.
\end{abstract}

\maketitle

\begin{section}{Introduction}

Finite-dimensional C$^*$-algebras are $ * $-isomorphic to finite direct sums of matrix algebras over the complex numbers.  Since linear algebra makes matrices tractable, it is not surprising that some of the earliest classes of C$^*$-algebras which were studied and classified were limits of finite-dimensional C$^*$-algebras.  For instance, we understand the compact operators on a Hilbert space, the uniformly hyperfinite or UHF algebras studied by Glimm \cite{MR0112057} which are classified by their supernatural number, and the approximately finite-dimensional or AF algebras studied by Bratteli \cite{MR0312282}, which are classified by their ordered $ K_0 $ group with order unit by Elliott's famous result \cite{MR0397420}. \par

Inspired by \cite{MR0229613}, we study a different type of limit of finite-dimensional C$^*$-algebras.  Namely, we study the logical limits of finite-dimensional C$^*$-algebras, which are called the pseudocompact C$^*$-algebras.  We also study the logical limits of matrix algebras, which we call the pseudomatricial C$^*$-algebras. \par

In Section 2 we give a crash course in continuous logic for the reader's convenience.  In Section 3 we recall several equivalent definitions of pseudocompact and pseudomatricial C$^*$-algebras due to Goldbring and Lopes.  We observe that the pseudocompact C$^*$-algebras are stable under direct sums, and both pseudocompact and pseudomatricial algebras are stable under tensoring by $ M_n $ for any $ n \in \mathbb{N} $, but in general pseudocompact C$^*$-algebras are not closed under taking subalgebras.  We recall the classification of the commutative pseudocompact C$^*$-algebras due to Henson and Moore.  In Section 4, we use axiomatizable properties from \cite{2016arXiv160208072F} to observe that the pseudocompact C$^*$-algebras are unital, stably finite, real rank zero, stable rank one, and tracial.  We also axiomatize the class of commutative pseudocompact C$^*$-algebras.  In Section 5 we explore the unitaries, projections, and centers of pseudocompact C$^*$-algebras.  We show that in pseudocompact C$^*$-algebras every unitary is homotopy equivalent to the identity, so these algebras have trivial $ K_1 $ groups.  We observe that if $ p $ is a projection in a pseudocompact C$^*$-algebra $ \A $ then $ p \A p $ is a pseudocompact C$^*$-algebra.  We show that every non-zero projection in a pseudocompact C$^*$-algebra dominates a minimal projection, and a non-zero projection $ p $ is minimal if and only if $ p \A p = \mathbb{C} p $.  Modulo some minimal projections, the identity can be written as a sum of $ d $ orthogonal Murray-von Neumann equivalent projections for any positive integer $ d $.  Pseudocompact C$^*$-algebras enjoy the Dixmier property, which allows us to observe that the center of a pseudocompact C$^*$-algebra is pseudocompact.  In Section 6 we focus on the subclass of pseudomatricial C$^*$-algebras.  We show that a pseudocompact C$^*$-algebra is pseudomatricial if and only if it has trivial center.  We show that pseudomatricial C$^*$-algebras have a unique tracial state, strict comparison of projections, and a totally ordered $ K_0 $ group with successors and predecessors.  The unique trace on an infinite-dimensional pseudomatricial C$^*$-algebra is not faithful, so such C$^*$-algebras are not simple.  We give some results on the structure of the $ K_0 $ group of a pseudomatrical C$^*$-algebra.  We conclude with some open questions about pseudocompact and pseudomatricial C$^*$-algebras. \par

We assume some familiarity with ultraproducts, especially the modified construction for Banach spaces.  For general information about ultraproducts of Banach spaces see \cite{MR552464}.  For ultraproducts of C$^*$-algebras, see \cite{MR1902808}.  Unless otherwise noted, the ultraproducts below are of the C$^*$-algebraic flavor. \par

The author would like to thank David Sherman and Isaac Goldbring for their support and helpful comments.

\end{section}

\begin{section}{Continuous Logic}

Unfortunately, classical true/false valued first-order logic does not work the way we would like for metric spaces like C$^*$-algebras.  This is because we have have elements which come close to satisfying an equation, without being able to exactly satisfy that equation.  For instance, in the rationals, $ x^2 - 2 $ has no root, but it has approximate roots of arbitrary precision, which is very different than the situation in the integers.  Thus $ x^2 - 2 $ has a root in a non-principal ultrapower of the rationals, but not in an ultrapower of the integers.  This makes the analogue of \L os' theorem (see below) fail.  \par

Fortunately, a continuous first-order logic and model theory for metric structures was introduced in \cite{MR2436146}, and specialized to C$^*$-algebras in \cite{MR3265292} and \cite{MR3210717}.  Another great reference is \cite{2016arXiv160208072F}.  \par

The language or symbols that can be used are the norm, addition, multiplication, multiplication by any complex scalar, and the adjoint operation.  Beside variables, we include the constant symbol for the zero element, and in unital C$^*$-algebras the identity $ I $.  Atomic formulae are norms of $ * $-polynomials in several variables.  One can combine (finitely many) formulae via continuous functions (our connectives), and can quantify by taking suprema and infima over the closed unit ball.  A formula with no free variables is called a sentence.  For a sentence $ \varphi $ and a C$^*$-algebra $ \A $, we let $ \varphi^\A $ be its evaluation in $ \A $, which is a real number. \par

Generally suprema are viewed as universal quantifiers and infima are viewed as existential quantifiers.  A sentence which evaluates to zero in a C$^*$-algebra is thought of as true in that algebra.  In general $ \inf_x \varphi( x ) = 0 $ does not mean this infimum is actually achieved, of course.  Note that $ | \varphi | + | \psi | $ evaluates to zero if and only if both $ \varphi $ and $ \psi $ evaluate to zero, so this operation acts like ``and''.  Similarly $ \varphi \cdot \psi $ evaluates to zero if and only if $ \varphi $ or $ \psi $ evaluates to zero, so this operation acts like ``or''.  \par

For example, in classical first-order logic one might express the fact that an algebra has a multiplicative unit with the sentence
$$
\exists e \quad \forall x \quad ( e x - x = 0 ) \wedge ( x e - x = 0 ) .
$$
Similarly one might express that an algebra is commutative with the sentence
$$
\forall x, y \quad x y - y x = 0 .
$$
The analogous sentences for C$^*$-algebras are
$$
\varphi_u = \inf_{ || e || \le 1 } \sup_{ || x || \le 1 } || e x - x || + || x e - x || 
$$
and
$$
\varphi_c = \sup_{ || x ||, || y || \le 1 } || x y - y x || .
$$
It is clear that if a C$^*$-algebra is unital then the first sentence evaluates to zero, and a C$^*$-algebra is commutative if and only if the second sentence evaluates to zero.  It is less clear that if $ \varphi_u = 0 $ then the C$^*$-algebra has an honest-to-goodness identity element instead of a sequence of elements which approximately behave like one.  Likewise, it is not immediately clear how small the evaluations of these sentences can be on non-unital and non-commutative C$^*$-algebras. \par

Similar to classical first-order logic, continuous logic does not allow us to quantify over arbitrary subsets.  However we can quantify over so-called definable sets.  See Section 9 in \cite{MR2436146} and \textsection 3.2 in \cite{2016arXiv160208072F} for an in-depth discussion.  It is well known (e.g. Lemma 3.2.4 in \cite{2016arXiv160208072F}) that the zero sets of weakly-stable formulae (see \cite{MR1420863}) are definable sets, and we can quantify over those sets.  In particular, results like Proposition 2.1 in \cite{MR1902808} imply the self-adjoint elements, the positive elements, the projections, the isometries and partial isometries, and the unitary elements form definable sets in any C$^*$-algebra.  Conversely, the normal, invertible, and central elements are not always definable sets.  By \textsection 2.3 in \cite{2016arXiv160208072F} we can quantify over certain compact subsets of the complex numbers.  For instance, if $ \psi $ is a formula with one free variable,
$$
\sup_{ 0 \le \lambda \le 1 } \psi( \lambda ) = \sup_{ || x || \le 1 } \psi( || x || ) .
$$

Two C$^*$-algebras $ \A $ and $ \mathfrak{B} $ are elementarily equivalent, written $ \A \equiv \mathfrak{B} $, if for every sentence $ \varphi $, $ \varphi^\A = \varphi^\mathfrak{B} $.  In other words, elementarily equivalent C$^*$-algebras are indistinguishable by continuous logic.  This is in general a coarser relation than $ * $-isomorphism.  See e.g. Theorem 3 in \cite{MR3151403}.  Two famous theorems show the close connection between elementary equivalence and ultrapowers.  The first is \L os' Theorem, which states that $ \A $ is elementarily equivalent to all of its ultrapowers: $ \A \equiv \A^\U $ for all ultrafilters $ \U $.  See Theorem 5.4 in \cite{MR2436146} and Proposition 4.3 in \cite{MR3265292}.  Conversely, we have the Keisler-Shelah Theorem, which states that two C$^*$-algebras are elementarily equivalent if and only if they have isomorphic ultrapowers (with respect to the same ultrafilter, even).  That is, $ \A \equiv\mathfrak{B} $ if and only if there is some ultrafilter $ \U $ so that $ \A^\U \cong \mathfrak{B}^\U $.  See Theorem 5.7 in \cite{MR2436146}. \par

\end{section}

\begin{section}{Definition of Pseudocompactness}

The term pseudofinite harkens back to \cite{MR0229613}.  An infinite field is pseudofinite if it satisfies every sentence which is satisfied in every finite field.   The analogous property in continuous logic is called pseudocompactness and was introduced in \cite{MR3373616}.  Here is a restatement of Lemma 2.4 in that paper:

The following are equivalent for a C$^*$-algebra $ \A $:
\begin{enumerate}

\item{}
Let $ \varphi $ be a sentence in the language of C$^*$-algebras.  If $ \varphi^\mathcal{F} = 0 $ for all finite-dimensional C$^*$-algebras $ \mathcal{F} $, then $ \varphi^\A = 0 $.

\item{}
Let $ \psi $ be a sentence in the language of C$^*$-algebras.   If $ \psi^\A = 0 $, then for all $ \varepsilon > 0 $ there is a finite-dimensional C$^*$-algebra $ \mathcal{F} $ so that $ | \psi^\mathcal{F} | < \varepsilon $.

\item{}
$ \A $ is elementarily equivalent to an ultraproduct of finite-dimensional C$^*$-algebras.

\end{enumerate}

\begin{defn*}
We say that $ \A $ is a \emph{pseudocompact} C$^*$-algebra if it satisfies any of the above conditions.  If one replaces ``finite-dimensional C$^*$-algebra'' with ``matrix algebra'' throughout the above, we say that $ \A $ is a \emph{pseudomatricial} C$^*$-algebra.  We do not require $ \A $ to be infinite-dimensional.
\end{defn*}

Pseudocompact tracial von Neumann algebras were studied in Section 5 of \cite{MR3210717}. \par

Notice that the first condition and \L os' theorem imply that an ultraproduct of pseudocompact C$^*$-algebras is pseudocompact.  Straightforward calculations show that taking direct sums and tensoring with matrix algebras commute with ultraproducts, so by appealing to the Keisler-Shelah theorem we obtain the following:

\begin{prop}
Pseudocompact C$^*$-algebras are closed under direct sums and tensoring with matrix algebras.  Similarly, pseudomatricial C$^*$-algebras are stable under tensoring with matrix algebras.
\end{prop}

We are specifically interested in the separable infinite-dimensional pseudocompact C$^*$-algebras.  Although ultraproducts are generally finite-dimensional or non-separable (if the ultrafilter is countably incomplete, see Theorem 5.1 and section 6 in \cite{MR1902808}), the downward  L\"{o}wenheim-Skolem theorem (Theorem 4.6 in \cite{MR3265292} and Proposition 7.3 in \cite{MR2436146}) allows us to take a separable elementarily equivalent subalgebra (in fact, an elementary subalgebra, see Definition 4.3 in \cite{MR2436146}), which will be an infinite-dimensional pseudocompact C$^*$-algebra.  \par

Pseudocompact real Banach spaces were studied by Henson and Moore, although under different terminology.  They showed that for a compact Hausdorff space $ K $, $ \mathcal{C}( K ) $ is pseudocompact (as a real Banach space) if and only if $ K $ is totally disconnected and has a dense subset of isolated points.  See Section 4 in \cite{MR0461104}, Theorem 4.1 in \cite{MR552464}, and \cite{MR620577}.  This result still holds for C$^*$-algebras.  Many of the relevant arguments can be found in \cite{MR3390013}. \par

For example, let $ \mathfrak{C} $ denote the usual Cantor set.  Since the Cantor set does not have a dense set of isolated points, $ \mathcal{C}( \mathfrak{C} ) $ is not pseudocompact.  In particular we can observe that not all commutative AF algebras (i.e. those $ \mathcal{C}( X ) $ with $ X $ totally disconnected and metrizable) are pseudocompact.  However, the space of convergent sequences of complex numbers, which is the space of continuous functions on the one-point compactification of the natural numbers is pseudocompact.  Likewise, $ \ell^\infty $, which is the space of continuous functions on the Stone-\v{C}ech compactification of the natural numbers, is also pseudocompact. \par

Since subalgebras of finite-dimensional algebras are finite-dimensional, one might expect subalgebras of pseudocompact algebras to be pseudocompact.  However, this is not the case. Let
$$
S_1 = \{ 0, 1 \}, \qquad S_2 = \{ 0,\tfrac{1}{3}, \tfrac{2}{3}, 1 \}, \qquad S_3 = \{ 0, \tfrac{1}{9}, \tfrac{2}{9}, \tfrac{ 1 }{ 3 }, \tfrac{ 2 }{ 3 } , \tfrac{7}{9}, \tfrac{8}{9}, 1 \} ,  \qquad \ldots .
$$
In general, we let $ S_n = \{ \tfrac{ k }{ 3^{n-1} } \, | \, k \in \mathbb{N} \} \cap \mathfrak{C} $ be the set of endpoints of the $ n^{ \text{th} } $ step of the usual middle-third construction of $ \mathfrak{C}$.  Define $ X_0 $ to be the space 
$$
X_0 = \{ \, ( 0, x ) \, | \, x \in \mathfrak{C} \} \cup \{ \, ( \tfrac{1}{n}, x ) \, | \, n \in \mathbb{N}, x \in S_n \} 
$$
with the subspace topology from $ \mathbb{R}^2 $.  This is compact, Hausdorff, totally disconnected, and has a dense subset of isolated points.  This will quotient onto every compact metric space by quotienting onto the Cantor set then following with the surjection that exists by the Hausdorff-Alexandroff theorem.  In particular, every $ \mathcal{C}( X ) $ space with $ X $ compact metric will be a subalgebra of the pseudocompact C$^*$-algebra $ \mathcal{C}( X_0 ) $.  Rudin showed that there is an interesting subclass of commutative pseudocompact C$^*$-algebras which are closed under subalgebras, see \cite{MR0085475}. \par

\end{section}

\begin{section}{Axiomatizable Properties}

A closed condition is of the form $ \varphi \le r $ for a sentence $ \varphi $ and $ r \in \mathbb{R} $.  A class $\mathcal{C} $ of C$^*$-algebras is axiomatizable if there is a collection of closed conditions $ \Sigma $ so that $ \mathcal{C} $ is the class of C$^*$-algebras satisfying those conditions.  We say that $ \Sigma $ is a set of axioms for that class.  We say that a property is axiomatizable if the class of C$^*$-algebras which enjoy that property is axiomatizable.  By Proposition 5.14 in \cite{MR2436146}, a class of C$^*$-algebras is axiomatizable if and only if it is closed under $*$-isomorphism, ultraproducts, and ultraroots (that is, for a C$^*$-algebra $ \A$ and an ultrafilter $ \u $, if $ \A^\u $ has the property, then $ \A $ has the property).  \par

Using this, it is easy to see that the pseudocompact C$^*$-algebras and the pseudomatricial C$^*$-algebras are the smallest axiomatizable classes containing the finite-dimensional C$^*$-algebras and matrix algebras respectively. \par

A paradigm to find properties of pseudocompact C$^*$-algebras is to identify axiomatizable properties of finite-dimensional C$^*$-algebras. \cite{2016arXiv160208072F} provides a useful catalogue of axiomatizable properties.  \par

The class of commutative C$^*$-algebras is axiomatized by the condition
$$
\varphi_c = \sup_{ || x ||, \, || y || \le 1 } || x y - y x || = 0 .
$$
In fact one can show that the above sentence evaluates to 2 in a non-commutative C$^*$-algebra.  
\begin{lemma}
Suppose $ \A $ is a non-commutative C$^*$-algebra.  Then for all $ \varepsilon > 0 $, there are elements $ a $, $ b $ of $ \A $ so that $ || a || $, $ || b || \le 1 + \varepsilon $ and $ || \, [ a, b ] \, || \ge 2 $.
\end{lemma}

\begin{proof}
Since $ \A $ is non-commutative, not every irreducible representation of $ \A $ acts on a one-dimensional Hilbert space.  Let $ \pi : \A \rightarrow \B( \H ) $ be an irreducible representation of $ \A $ where $ \H $ is at least two-dimensional.  Let $ \xi $, $ \eta $ be orthogonal unit vectors in $ \H $, and let $ \K = \text{span} ( \xi, \eta ) $.  Then there are operators $ t_1 $, $ t_2 \in \B( \H ) $ so that
$$
t_1 \xi = - \eta \qquad t_1 \eta = \xi \qquad t_2 \xi = \eta \qquad t_2 \eta = \xi \qquad || t_1 || = || t_2 || = 1 .
$$
Then by the Kadison Transitivity Theorem, there are $ a $, $ b \in \A $ such that 
$$
a|_\K = t_1 |_\K, \qquad b|_\K = t_2|_\K \qquad || a || , \, || b || \le 1 + \varepsilon .
$$
Then since
$$
[ a, b ] \xi = ( a b - b a ) \xi = ( t_1 t_2 - t_2 t_1 ) \xi = t_1 t_2 \xi - t_2 t_1 \xi = t_1 \eta + t_2 \eta = 2 \xi .
$$
We have that $ || \, [ a, b ] \, || \ge 2 $.
\end{proof}

This means that the class of non-commutative C$^*$-algebras is axiomatizable.  For an alternative sentence and proof, see \textsection 2.5(a) in \cite{2016arXiv160208072F}. \par

The class of unital C$^*$-algebras is axiomatized by the condition
$$
\varphi_u = \inf_{ || e || \le 1 } \sup_{ || x || \le 1 } || e x - x || = 0 .
$$
One can show the above sentence evaluates to at least 1 in a non-unital C$^*$-algebra.

Recall that a C$^*$-algebra $ \A $ is real rank zero if the self-adjoint elements with finite spectra are dense in the self-adjoint elements of $ \A $, see \cite{MR1120918}.  Having real rank zero is axiomatizable by a sentence $ \varphi_{ \text{ rr0} } $, see Example 2.4.2 and \textsection 3.6(b) in \cite{2016arXiv160208072F}.  This and the Henson-Moore classification allow us to axiomatize the commutative pseudocompact C$^*$-algebras. \par

\begin{prop}
$ \A $ is a commutative pseudocompact C$^*$-algebra if and only if it is commutative, unital, real rank zero, and every element can be approximately normed by a minimal projection $ p $ with $ p \A p = \mathbb{C} p $.  The following form a set of axioms for the class of commutative pseudocompact C$^*$-algebras: \par

\begin{enumerate}

\item{}
$ \varphi_c^\A = \sup_{ || x ||, \, || y || \le 1 } || x y - y x ||  = 0 $.

\item{}
$ \varphi_u^\A = \inf_{ || e || \le 1 } \sup_{ || x || \le 1 } || e x - x || = 0 $.

\item{}
$ \varphi_{ \text{ rr0} }^\A = 0 $.

\item{}
$ \displaystyle \sup_{ || x || \le 1 } \inf_{ p \text{ a proj. } } \sup_{ || y || \le 1 } \inf_{ | \lambda | \le 1 } || p y p - \lambda p ||+ \big| \, || x || - || x p || \, \big| = 0 $.

\end{enumerate}

\end{prop}

\begin{proof}

The first two axioms guarantee that the C$^*$-algebra is unital and commutative, so it is of the form $ \mathcal{C}( X ) $ with $ X $ compact Hausdorff.  \par

Recall that a commutative unital C$^*$-algebra $ \A \cong \mathcal{C}( X ) $ is real rank zero if and only if $ X $ is totally disconnected.  Thus we just need to justify that every element being approximately normed by a minimal projection is equivalent to $ X $ having a dense set of isolated points. \par

Suppose the underlying compact Hausdorff space $ X $ has a dense subset of isolated points.  Let $ \varepsilon > 0 $ and $ f \in \mathcal{C}( X ) $ be given, then $ \{ x \, : \, | f( x ) | > || f || - \varepsilon \} $ is a non-empty open set, so it contains an isolated point $ x_0 $, so $ p = \chi_{ \{ x_0 \} } $ is a minimal projection (and $ p \A p = \mathbb{C} p $) and $ || p f || = | f( x_0 ) | > || f || - \varepsilon $. \par

Conversely, suppose that $ X $ does not have a dense subset of isolated points.  By Urysohn's lemma there is a function that vanishes on the closure of the isolated points but has norm 1, and this shows that the axiom above has value at least 1.
\end{proof}

Recall that a unital C$^*$-algebra $ \A $ is stable rank one if the invertible elements are dense in $ \A $, see 3.3 in \cite{MR693043}.  An explicit axiomatization of stable rank one is given in Lemma 3.7.2 in \cite{2016arXiv160208072F}. \par

For unital C$^*$-algebras, being finite and being stably finite are axiomatizable properties, see \textsection 3.6(d) in \cite{2016arXiv160208072F}.  \par

For unital C$^*$-algebras, having a tracial state is an axiomatizable property, see \textsection 2.5(f) and \textsection 3.5(a) in \cite{2016arXiv160208072F}.  If $ \tau_i $ are tracial states on $ \A_i $ then $ \lim_\u \tau_i $ is a tracial state on $ \prod_\u \A_i $.  On the other hand if $ \tau $ is a tracial state on $ \A^\u $ then we can get a tracial state $ \tau_0 $ on $ \A $ defined by $ \tau_0( a ) = \tau( \, ( a )_\u \, ) $.  Traces on ultraproducts are studied in \cite{bicefarahtraces}. \par

We collect the above observations:

\begin{prop}
Pseudocompact C$^*$-algebras are unital, real rank zero, stable rank one, stably finite, and tracial.
\end{prop}

\end{section}

\begin{section}{Properties of Pseudocompact C$^*$-Algebras}

In this section, we utilize unitaries, projections, and central elements to study the properties of pseudocompact C$^*$-algebras.

\begin{prop}
In pseudocompact C$^*$-algebras, every unitary is homotopic to the identity.  In particular, the $ K_1 $ group of a pseudocompact C$^*$-algebra is trivial.
\end{prop}

\begin{proof}
Recall that a pseudocompact C$^*$-algebra has an ultrapower isomorphic to an ultraproduct of finite-dimensional C$^*$-aglebras.  In a finite-dimensional C$^*$-algebra, every unitary is homotopic to the identity, and in fact every unitary is of the form $ \exp( 2 \pi i t ) $ for some self-adjoint contraction $ t $.  Every unitary in an ultraproduct of finite-dimensional algebras has a representing sequence of unitaries (see Proposition 2.1 in \cite{MR1902808}), and $ \exp( ( x_n )_\u ) = ( \exp( x_n )_\u ) $ as the exponential function is a uniform limit of polynomials on the closed ball of radius $ || ( x_n )_\u || $.  Thus we see that every unitary in an ultraproduct of finite-dimensional C$^*$-algebras is also an exponential of a skew-self-adjoint element.  If every unitary element in an ultrapower $ \A^\u $ is an exponential of a skew-self adjoint, then every unitary in $ \A $ is a norm-limit of exponentials of skew-self-adjoint elements.   Recall that if two unitaries are within distance two of each other, they are homotopic, see 2.3.1 in \cite{MR1656031}.  Thus every unitary in a pseudocompact algebra is the norm-limit of exponentials of a skew-self-adjoint element, and thus is homotopic to the identity.
\end{proof}

Given that finite-dimensional C$^*$-algebras are determined by matrix units, and the equations defining matrix units are weakly stable, it is not surprising that they are our main tool for understanding pseudocompact C$^*$-algebras.  For instance, using the fact that a projection in an ultrapower has a representative sequence of projections (see Proposition 2.1 in \cite{MR1902808}) and the Keisler-Shelah theorem, one can show the following:

\begin{prop}
If $ p $ is a projection in a pseudocompact C$^*$-algebra $ \A $ then $ p \A p $ is pseudocompact.
\end{prop}

In other words, a corner of a pseudocompact C$^*$-algebra is pseudocompact.  The next result is an analogue of the fact that minimal projections in finite-dimensional C$^*$-algebras are exactly the rank one projections.

\begin{prop}
In a pseudocompact C$^*$-algebra $ \A $, a (non-zero) projection $ p $ is minimal if and only if $ p \A p = \mathbb{C} p $.
\end{prop}

\begin{proof}
This is an axiomatizable property:
\begin{equation*}
\sup_{ \substack{ p \ne 0 \\ \text{ proj. } } }  \left( \sup_{ || a || \le 1 } \inf_{ | \lambda | \le 1 } || p a p - \lambda p || \,  \right) \cdot \left( \inf_{ \substack{ q \ne 0 \\ \text{ proj. } } } || p q - q || + \big| \, || p - q || - 1 \big| \, \right) = 0 . \qedhere
\end{equation*}
\end{proof}

Using Proposition 2.1 in \cite{MR1902808}, it is straight-forward to check that a projection $ p $ in an ultraproduct $ \prod_\U \A_i $ is minimal if and only if it has a representative sequence $ p = ( p_i )_\U $ so $ \{ i \, | \, \text{ $ p_i $ is a minimal projection in } \A_i \} \in \U $.  Since a pseudocompact C$^*$-algebra has an ultrapower isomorphic to an ultraproduct of finite-dimensional C$^*$-algebras, pseudocompact C$^*$-algebras have many minimal projections.

\begin{prop}
In a pseudocompact C$^*$-algebra $ \A $, every non-zero projection dominates a minimal projection.
\end{prop}

\begin{proof}
This is an axiomatizable property:
$$
\sup_{  \substack{ q \ne 0 \\ \text{ proj. } } } \inf_{ \substack{ p \ne 0 \\ \text{ proj. } } } \sup_{ || a || \le 1 } \inf_{ | \lambda | \le 1 } || p a p - \lambda p || +  || q p - p || = 0 .
$$
In pseudomaticial C$^*$-algebras, the existence of minimal projections allows us to use
\begin{equation*}
\inf_{ \substack{ p \ne 0 \\ \text{ proj. } } } \sup_{  \substack{ q \ne 0 \\ \text{ proj. } } } \inf_{ v \text{ partial isom. } } || v^* v - p || + || q ( v v^* ) - v v^* || = 0 . \qedhere
\end{equation*}
\end{proof}

Notice that this is quite different than the tracial von Neumann algebra case considered in Section 4 of \cite{MR2844455}, Proposition 6.5 in \cite{MR3265292}, and Section 5 of \cite{MR3210717}.  \par

Since UHF algebras lack minimal projections, we observe \par

\begin{prop}
UHF algebras are not pseudocompact.
\end{prop}

Let $ n $ and $ d $ be natural numbers.  Then we can write $ n = k d + r $ where $ 0 \le r < d $.  This means the identity in $ M_n $ is the orthogonal sum of $ d $ Murray-von Neumann equivalent projections (each of rank $ k $) and fewer than $ d $ minimal, (i.e. rank one) projections.  We can similarly decompose the identity in any pseudocompact C$^*$-algebra:

\begin{prop}
Let $ d \ge 2 $ be a natural number.

\begin{enumerate}

\item{}
In a pseudomatrical C$^*$-algebra, the identity can be written as a sum of $ d $ orthogonal Murray-von Neumann equivalent projections plus $ d - 1 $ orthogonal minimal or zero projections.

\item{}
In a pseudocompact C$^*$-algebra, the identity can be written as a sum of $ d $ orthogonal Murray-von Neumann equivalent projections plus $ d - 1 $ orthogonal abelian (or zero) projections.

\end{enumerate}

\end{prop}

\begin{proof}
This is an axiomatizable property.  Here is the case $ d = 2 $:
\[
\inf_{ v \text{ a partial isom. } } \inf_{ p \text{ a proj }  } || I - ( v^* v + v v^* + p ) \, || + \sup_{ || x || , || y || \le 1 } || p x p y p - p y p x p || = 0 . \qedhere
\]
\end{proof}

When we write the identity as $ I = r + \sum_{ i = 1 }^d p_i $ where the $ p_i $ are orthogonal, pairwise Murray-von Neumann equivalent projections and $ r $ is a sum of $ d - 1 $ or fewer orthogonal minimal projections, we call a $ p_i $ an approximate $ 1/d^{ \text{th} } $ of the identity.

\begin{prop}
The unitization of the compact operators on an infinite-dimensional Hilbert space is not pseudocompact.
\end{prop}

\begin{proof}
In the unitization of the compacts, projections are either finite-dimensional or co-finite-dimensional, so there are no approximate halves of the identity.
\end{proof}

\begin{conj}
$$
\prod_\u M_{ n_i } \equiv \prod_{ \V } M_{ m_j } \qquad \text{ if and only if } \qquad \lim_\u n_i \bmod d = \lim_{ \V } m_i \bmod d \quad \text{ for all } d \in \mathbb{N} 
$$
\end{conj}

If the condition on the right fails, the resulting pseudocompact C$^*$-algebras are not elementarily equivalent and thus not $ * $-isomorphic. \par

\begin{cor}
There are uncountably many $ * $-isomorphism classes of separable pseudomatrical C$^*$-algebras.
\end{cor}

Recall we say a unital C$^*$-algebra $ \A $ with center $ \Z( \A ) $ has the Dixmier property if for all $ a \in \A $, the norm-closed convex hull of the unitary orbit of $ \A $ contains a central element:
$$
\text{For all } a \in \A, \qquad \overline{ \conv \{ u a u^* \, : \, u \in \mathcal{U}( \A ) \} } \cap \Z( \A ) \ne \emptyset .
$$
Let $ \mathfrak{M} $ be a von Neumann algebra.  For every self-adjoint $ x \in M $, there is a $ u \in \mathcal{U}( \mathfrak{M} ) $ and a $ z \in \Z( \mathfrak{M} ) $ with
$$
|| \, \tfrac{1}{2} ( x + u x u^* ) - z || \le \frac{ 3 }{ 4 } || x || .
$$
Iterating this result shows that finite von Neumann algebras have the Dixmier property, see Chapter 5, Section 3 of \cite{MR1451139}, or Theorems III.2.5.18 and 19 in \cite{MR2188261}.

\pagebreak

\begin{prop}
Pseudocompact C$^*$-algebras have the Dixmier property.
\end{prop}

\begin{proof}
For each $ n $ pick $ K_n \in \mathbb{N} $ satisfying $ ( 3/4 )^{ K_n } < 1/2n $.  Take $ \{ \lambda_1, \ldots, \lambda_{ D_n } \} $ to be a $ \tfrac{ 1 }{ 2n } $ net in the unit disk in the complex numbers.  Now for any finite-dimensional C$^*$-algebra, since it is a finite von Neumann algebra, for all $ n $, and all self-adjoint $ x $, there is a central element $ z $ so that
$$
\sup_x \quad \inf_{ \substack{ u_1, \ldots, u_{ K_n } \\ \text{ unitaries } } } \quad \left| \left| \sum_{ i = 1 }^{ K_n } \frac{ 1 }{ 2^{ K_n } } u_i^* x u_i - z \right| \right| < \frac{ 1 }{ 2n } .
$$
Since we are in a finite-dimensional C$^*$-algebra, the central element $ z $ is a linear combination of central projections, $ z = \sum_1^d \mu_i q_i $, which is within distance $ 1/2n $ to a linear combination of the form $ \sum_1^{ D_n } \lambda_{ i } p_i $ where the $ p_i $ are central projections.  Thus for all $ n $, in every finite-dimensional C$^*$-algebra,
$$
\sup_x \quad \inf_{ \substack{ u_1, \ldots, u_{ K_n } \\ \text{ unitaries } } } \quad \inf_{ \substack{ p_1, \ldots, p_{ D_n } \\ \text{ central projs. } } } \quad \inf_{ \lambda_1, \ldots , \lambda_{ D_n } \in \overline{ \mathbb{D} } } \quad \left| \left| \sum_{ i = 1 }^{ K_n } \frac{ 1 }{ 2^{ K_n } } u_i^* x u_i - \sum_{ j = 1 }^{ D_n } \lambda_j p_j \right| \right| \le \frac{ 1 }{ n } .
$$
In particular, we have the same inequality for every pseudocompact C$^*$-algebra.
\end{proof}

Let $ \A $ be a C$^*$-algebra with center $ \Z( \A ) $.  For each $ a \in \A $ we have the inner derivation induced by $ \A $, $ \Delta_a : x \mapsto a x - x a $.  When $ \A $ is a finite-dimensional C$^*$-algebra, for each $ a \in \A $ we have $ || \Delta_a || = 2 \dist( a, \Z( \A ) ) $, see \cite{MR0326412}.  In \cite{MR0482236}, Archbold defined 
$$
\K( \A ) = \inf \{ K \, : \,  \dist( a, \Z( \A ) ) \le K || \Delta_a || \, \forall a \in \A \} .
$$

\begin{cor}
~
\begin{enumerate}
\item{}
Let $ \A $ be a pseudocompact C$^*$-algebra.  Then $ K( \A ) \le 1 $ and the inner derivations on $ \A $ are point-norm closed in the space of all derivations on $ \A $.
\item{}
If $ \A_i $ is pseudocompact for all $ i \in I $, and $ \u $ is an ultrafilter on $ I $, we have $ \Z( \prod_\u \A_n ) = \prod_\u \Z( \A_n ) $.  Pseudomatrical C$^*$-algebras have trivial center.  The center of a pseudocompact C$^*$-algebra is pseudocompact.
\end{enumerate}
\end{cor}

\begin{proof}
For unital C$^*$-algebras, the Dixmier property implies that $ K( \A ) \le 1 $.  See Section 2 of \cite{MR0482261}.  By Theorem 5.3 in \cite{MR0215111} this implies that the inner derivations are point-norm closed.  Since $ K( \A ) \le 1 $, those $ a $ which approximately commute with everything in the unit ball are close to the center of $ \A $: $ \dist( a, \Z( \A ) ) \le || \Delta_a || $.  This gives us that the center of the ultraproduct of pseudocompacts is the ultraproduct of the centers.  The final claim follows from the Keisler-Shelah theorem.
\end{proof}

Note that $ \K( \A ) $ can be infinite even for AF-algebras. Example 6.2 in \cite{MR0482236} gives a unital AF algebra $ \A $ with trivial center and a bounded sequence $ a_n \in \A $ with $ \lim_{ n \rightarrow \infty } || \Delta_{ a_n } || = 0 $ but $ \dist( a_n, \Z( \A ) ) = 1 $ for all $ n $. \par

Recall that a separable C$^*$-algebra is an MF algebra if it can be written as an inductive limit of a generalized inductive system of finite-dimensional C$^*$-algebras.  See \cite{MR1437044}, \cite{MR0164248} and \cite{MR0500178}.  By 11.1.5 in \cite{MR2391387}, an algebra is MF if and only if it admits norm microstates.  Using norm-microstates and Theorem 4.1 in \cite{MR1902808}, one gets the following classification: A separable C$^*$-algebra is MF if and only if it is a subalgebra of a pseudocompact C$^*$-algebra.  This result was communicated to the author by Ilijas Farah. \par

The list of properties of pseudocompact C$^*$-algebras is remarkably similar to the properties of the C$^*$-algebras studied in \cite{MR2063121} and \cite{MR3010147}, except those algebras are simple.  David Sherman observed that pseudocompact C$^*$-algebras are elementarily equivalent to their opposite algebras.  Proposition 1.1 in \cite{MR3404668} shows that the von Neumann algebras which are elementarily equivalent to their opposite algebras form an axiomatizable class.  The same argument holds for C$^*$-algebras.  Since finite-dimensional C$^*$-algebras are isomorphic to their opposite algebras, and the class of C$^*$-algebras which are elementarily equivalent to their opposite algebra is an axiomatizable class, the pseudocompact C$^*$-algebras are elementarily equivalent to their opposite algebras. \par

\end{section}

\begin{section}{Properties of Pseudomatricial C$^*$-algebras}

Now we will focus our attention on the smaller class of pseudomatricial C$^*$-algebras.  \par

Infinite-dimensional pseudomatricial C$^*$-algebras are never nuclear, nor elementarily equivalent to a nuclear C$^*$-algebra, see Proposition 7.2.4 in \cite{2016arXiv160208072F}. \par

Pseudomatrical C$^*$-algebras are easily distinguished from other pseudocompact algebras since they have trivial center. \par 

\begin{prop}
A pseudocompact C$^*$-algebra $ \A $ is pseudomatricial if and only if $ \Z( \A ) = \mathbb{C} I $.
\end{prop}

\begin{proof}
The pseudocompact C$^*$-algebras which are not pseudomatricial will be elementarily equivalent to $ \prod_\u F_i $ where the $ F_i $ are finite-dimensional and $ \u $-many of the $ F_i $ are non-trivial direct sums of matrix algebras.  In particular, $ \u $-many of the $ F_i $ have non-trivial central projections, axiomatized by
\begin{equation*}
\inf_{ p \text{ a proj. } }  \big| \, || p - I || - 1 \, \big| + \big| \, || p || - 1 \, \big| + \sup_{ || x || \le 1 } || p x - x p || = 0 . \qedhere
\end{equation*}
\end{proof}

\begin{prop} \label{comparablelemma}
In a pseudomatricial C$^*$-algebra, the projections are totally ordered by Murray-von Neumann subequivalence.
\end{prop}

\begin{proof}
The property that ``all projections are comparable'' is axiomatized by the following sentence:
\[
\sup_{ p, q \text{ projs } } \inf_{ x \text{ partial isom. } }  ( \, || p - x^* x || + || ( x x^* ) q - x x^* || \, ) \cdot ( \, || q - x^* x || + || ( x x^* ) p - x x^* || \, ) = 0 . \qedhere
\]
\end{proof}

\begin{cor}
If $ \A $ is a pseudomatricial C$^*$-algebra, then $ K_0( \A ) $ is totally ordered and has successors and predecessors.  If $ \A $ is separable, then $ K_0( \A ) $ is a countable abelian totally ordered group, so it is a dimension group (see \cite{MR564479}).
\end{cor}

\begin{proof}
If $ m $ is a (non-zero) minimal projection then $ g - [ m ]_0 $ and $ g + [ m ]_0 $ are the greatest element less than $ g $ and the least element greater than $ g $, respectively.
\end{proof}

\pagebreak

\begin{cor}
Let $ \A $ be an infinite-dimensional pseudomatricial C$^*$-algebra.  For every minimal projection $ p $ and any tracial state $ \tr $, $ \tr( p ) = 0 $.  In particular, infinite-dimensional pseudomatrical algebras have a non-faithful tracial state, and they are not simple.
\end{cor}

\begin{proof}
The identity $ I $ dominates a minimal projection $ m_1 $.  Then since $ \A $ is infinite-dimensional, $ I - m_1 $ is non-zero and dominates another minimal projection $ m_2 $ orthogonal to $ m_1 $.  Continuing in this way, for all $ n \in \mathbb{N} $ we can iteratively find $ n $ orthogonal minimal projections $ m_i $.  Since all minimal projections are Murray-von Neumann equivalent, $ \tr( m_i ) = \tr( m_j ) $ for all $ i, j \in \mathbb{N} $.  Thus
$$
1 = \tr( I ) \ge \tr \left( \sum_{ i = 1 }^n m_i \right) = n \tr( m_1 ) .
$$
So for all $ n $, the trace of a minimal projection is less than $ 1/n $, so minimal projections must have trace zero.
\end{proof}

Minimal projections in pseudomatrical C$^*$-algebras behave like infinitesimal elements: they are non-zero and have norm one, but they are subequivalent to every projection with non-zero trace.

\begin{prop}
Pseudomatrical C$^*$-algebras have a unique tracial state.
\end{prop}

\begin{proof}
Since a pseudomatricial C$^*$-algebra has real rank zero, the span of the projections is dense.  Thus the trace is determined by its value on projections, and the trace on each projection is determined by how many orthogonal Murray-von Neumann equivalent copies of the projection (or approximate fractions of the projection) one can find.

More precisely, the maximum number of orthogonal Murray-von Neumann equivalent copies of $ p $ is well-defined by finiteness and cancellation.  If there are $ n $ orthogonal Murray-von Neumann equivalent copies of $ p $, then $ \tr( p ) \le 1/n $.  Now if $ q $ is an approximate $ d^{\text{th} } $ of $ p $ then $ \tr( q ) = \tr( p ) / d $.  Repeating the process with smaller and smaller fractions of $ p $ determines the trace of $ p $ as precisely as we wish.
\end{proof}

Note that the trace ideal $ J $ of a separable, infinite-dimensional pseudomatricial C$^*$-algebra $ \A $ contains a copy of the compacts (generated by any countable collection of orthogonal minimal projections -- see 7.1.2 in \cite{MR1656031}), and the quotient $ \A / J $ has matrix units of all orders. 

\begin{prop}
Pseudomatrical algebras have the strong Dixmier property, i.e., for all $ a \in \A $, $ \overline{ \conv \big( \, \U( a ) \, \big) } \cap \Z( \A ) $ is a singleton.
\end{prop}

\begin{proof}
Note if $ z \in \overline{ \conv \big( \, \mathcal{U}( a ) \, \big) } $, then $ \tr( z ) = \tr( a ) $ .  Thus if $ z $ is in the center of a pseudomatricial C$^*$-algebra, $ \mathbb{C} I $, $ z = \tr(a) I $ is unique.
\end{proof}

\begin{cor}
If $ J $ is a (closed, two-sided) ideal of a pseudomatricial C$^*$-algebra $ \A $, and there is an $ x \in J $ with $ \tr( x ) \ne 0 $, then $ J = \A $.
\end{cor}

\begin{proof}
Since $ x \in J $ is a closed, two-sided ideal, $ \tr( x ) I \in \overline{ \conv \big( \, \mathcal{U}( x ) \, \big) }^{ || \cdot || } \subseteq J $.
\end{proof}

\pagebreak

\begin{prop}
Pseudomatricial C$^*$-algebras have strict comparison of projections.  That is, if $ \tr( q ) < \tr( p ) $ then $ q $ is equivalent to a proper subprojection of $ p $.
\end{prop}

\begin{proof}
Obviously this property holds in $ M_n $ since the trace of a projection is the rank of the projection divided by $ n $, and we have that $ q $ is subequivalent to $ p $ if and only $ \text{rank}( q ) \le \text{rank}( p ) $.  Suppose each $ \A_i $ has strict comparison with respect to $ \tr_i $.  Then $ \prod_\U \A_i $ has strict comparison of projections with respect to $ \tr = \lim_\U \tr_i $:  Suppose $ p $ and $ q $ are projections in $ \prod_\U \A_i $ with $ \tr( p ) < \tr( q ) $.  Without loss of generality, there are representative sequences of projections $ p = ( p_i )_\U, q = ( q_i )_\U $.  Since $ \tr( p ) = \lim_\U \tr_i ( p_i ) < \lim_\U \tr_i ( q_i ) = \tr( q ) $,
$$
S = \{ i \, | \, \tr_i( p_i ) < \tr_i( q_i ) \} \in \U .
$$
Thus for $ i \in S $, $ p_i $ is properly Murray-von Neumann equivalent to a subprojection $ q_i' $ of $ q_i $ via some partial isometry $ v_i $.  Considering the partial isometry $ ( v_i )_\U $, we see that $ p $ is properly Murray-von Neumann subequivalent to $ q $.  \par

Suppose $ \A^\U $ has strict comparison of projections with respect to $ \lim_\U \tr $ for some tracial state $ \tr $ on $ \A $.  Then $ \A $ has strict comparison of projection with respect to $ \tr $:  Suppose $ p, q $ are projections in $ \A $ with $ \tr( p ) < \tr( q ) $.  Then $ P = ( p )_\U $ and $ Q = ( q )_\U $ are projections in $ \A^\U $ with $ \lim_\U \tr ( p )_\u < \lim_\U \tr( q )_\U $.  So by assumption, $ P $ is properly Murray-von Neumann subequivalent to $ Q $.   So there is a partial isometry $ V $ with $ V^* V =  P $ and $ V V^* = Q' \le Q $.  Without loss of generality, representative sequences of projections $ q_i' \le q $ so $ Q' = ( q_i' )_\U $, and we can find partial isometries $ v_i $ so $ V = ( v_i )_\U $ and on large set of indices, $ v_i^* v_i = p $ and $ v_i v_i^* = q_i' \le q $.  So $ p $ is Murray-von Neumann subequivalent to $ q $.
\end{proof}

The converse is not true because there are non-zero trace zero projections. \par

The same argument shows that pseudocompact C$^*$-algebras have strict comparison of projections: if $ p $ and $ q $ are projections and for all traces $ \tr $, $ \tr( q ) < \tr( p ) $, then $ q $ is Murray-von Neumann subequivalent to $ p $. \par

In $ M_n $, Murray-von Neumann equivalence, unitary equivalence, and homotopy equivalence of projections are all equivalent.  The same property holds for pseudomatricial C$^*$-algebras:

\begin{prop}
In a pseudomatricial algebra, Murray-von Neumann equivalence, unitary equivalence, and homotopy equivalence are equivalent.
\end{prop}

\begin{proof}
It is clear unitary equivalence implies Murray-von Neumann equivalence.  We show that ``Murray-von Neumann equivalence implies unitary equivalence'' is an axiomatizable property.  Of course this property is preserved by isomorphism.  Suppose in $ \A_i $ Murray-von Neumann equivalence implies unitary equivalence.  Let $ p \sim q $ be Murray-von Neumann equivalent projections in $ \prod_\U \A_i $.  Without loss of generality we have that $ p = ( p_i )_\U $, $ q = ( q_i )_\U $, with $ p_i $ and $ q_i $ projections, and a partial isometry $ v = ( v_i )_\U $ so that
$$
S = \{ i \, | \, p_i = v_i^* v_i \text{ and } q_i = v_i v_i^* \} \text{ is in } \U .
$$
Thus for $ i \in S $, $ p_i $ and $ q_i $ are unitarily equivalent in $ \A_i $, so there is a unitary $ u_i \in \A_i $ so that $ u_i^* p_i u_i = q_i $.  For $ i \notin S $ let $ u_i = I_i $.  Letting $ u = ( u_i )_\U $ we have that $ u $ is a unitary and $ u^* p u = q $.  Thus $ p $ and $ q $ are are unitarily equivalent in $ \prod_\U \A_i $.

Suppose Murray-von Neumann equivalence implies unitary equivalence in $ \A^\U $.   Suppose $ p $ and $ q $ are Murray-von Neumann equivalent projections in $ \A $ via a partial isometry $ v $, then $ ( p )_\U $ and $ ( q )_\U $ are Murray-von Neumann equivalent in $ \A^\U $ via $ ( v )_\U $.  Then there is a unitary $ u \in \A^\U $ so $ u^* ( p )_\U u = ( q )_\U $.  Without loss of generality $ u = ( u_i )_\U $ where the $ u_i $ are unitaries in $ \A $ and $ ( u_i p u_i^* )_\U = ( q )_\U $, then
$$
S = \{ i \, : \,  || u_i p u_i^* - q || < 1 \} \text{ is in } \U .
$$
If two projections are distance less than one apart, they are unitarily equivalent (See e.g. 2.2.4 and 2.2.6 in \cite{MR1656031}).  Thus for $ i \in S $, $ u_i p u_i^* $ is unitarily equivalent to $ q $.  So $ p $ and $ q $ are unitarily equivalent.

It is clear that homotopy equivalence implies Murray-von Neumann equivalence.  By the proof of Proposition 5.1, the exponential unitaries are dense in the unitary group of a pseudomatricial C$^*$-algebra.  Since unitary equivalence via a unitary in the connected component of the identity implies homotopy equivalence (see 2.2.6 in \cite{MR1656031}), we see that unitary equivalence is equivalent to homotopy equivalence in pseudocompact C$^*$-algebras.
\end{proof}

Notice that if $ \A $ is pseudomatricial, then for any natural number $ n $, $ M_n( \A  ) $ is pseudomatricial.  The induced trace on $ M_n( \A ) $ from $ \A $ is just $ n $ times the unique tracial state on $ M_n( \A  ) $.  Elements of $ K_0( \A ) $ are of the form $ [ p ]_0 - [ q ]_0 $ where $ p, q $ are projections in some $ M_n( \A ) $, so either $ q \preceq p $ or vice-versa.  In the first case, $ p $ is the orthogonal sum of a projection $ p' $ and a projection $ q' $ which is Murray-von Neumann equivalent to $ q $.  Thus $ [ p ]_0 - [ q ]_0 = [ p' + q' ]_0 - [ q ]_0 = [ p' ]_0 + [ q' ]_0 - [ q ]_0 = [ p' ]_0 $.  So $ K_0( \A ) = \{ \pm [ p ]_0 \, | \, p \in M_n( \A ) \} $. \par

The $ K $-theory of ultraproducts has been studied in \cite{MR2180648}: the $ K_0 $ group is a sort of graded ultraproduct of the $ K_0 $ groups of the components, since a projection in $ K_0( \prod_\U \A_i ) $ is in some $ M_n( \prod_\U \A_i ) = \prod_\U M_n( A_i ) $.  In other words, the matrix amplifications need to be bounded. \par

Let $ G $ be a (totally) ordered additive group, $ g, h \in G $.  We let $ | g | $ denote the element of $ \{ g, -g \} $ which is greater than or equal to the zero element.  Recall $ g $ is Archimedean less than $ h $, denoted $ g \ll h $ if $ n | g | \le | h | $ for all natural numbers $ n $.  We say $ g, h \in G $ are Archimedean equivalent, denoted $ g \approx h $, if there are natural numbers $ n, m $ so that $ n | g | \ge | h | $ and $ m | h | \ge | g | $.

\begin{prop}
Let $ \A $ be a pseudomatricial C$^*$-algebra.  Then $ \ker( K_0( \tr ) ) $ is the subgroup of $ K_0( \A ) $ generated by the trace-zero projections, and it is a subgroup of $ \mathbb{R}^\eta $, real-valued functions from well-ordered subsets of $ \eta $, the set of Archimedean equivalence classes of trace zero projections, equipped with the lexicographical ordering.
\end{prop}

\begin{proof}
It is clear that $ \ker( K_0( \tr ) ) = \{ \, \pm [ p ]_0 \, | \, \tr( p ) = 0 \} $ is a subgroup of $ K_0 $.  Since $ \ker( K_0( \tr ) ) $ is a totally ordered abelian group, by the Hahn Embedding Theorem \cite{hahn},  $ \ker( K_0( \tr ) ) $ is a subgroup of $ \mathbb{R}^\eta $, the group of functions from well-ordered subsets of $ \eta $ into $ \mathbb{R} $, where $ \eta $ is the set of Archimedian equivalence classes of trace zero projections, equipped with the lexicographical ordering.  Note that if $ \A $ is separable, these groups are countable, so $ \eta $ is countable.
\end{proof}

\begin{prop}
Let $ \A $ be a pseudomatricial C$^*$-algebra.  Then $ K_0( \A ) \cong G \oplus \ker( K_0( \tr ) ) $ as ordered abelian groups, where $ G $ is a divisible subgroup of $ \mathbb{R} $, equipped with the usual lexicographical order.
\end{prop}

\begin{proof}
 Let $ \A $ be a pseudomatricial C$^*$-algebra, then 
$$
K_0( \A ) \cong K_0( \A ) / \ker( \, K_0( \tr ) \, ) \oplus \ker( \, K_0( \tr ) \, ) .
$$
We just need to show that $ K_0( \A ) / \ker( K_0( \tr ) ) $ is a subgroup of $ \mathbb{R} $.  This group is isomorphic to the image of $ K_0( \tr ) $, which is just all rational multiples of trace on projections in $ \A $.  If $ \lambda $ is the trace of a projection $ p $ in $ \A $ then $ n \lambda $ is the trace of $ \diag( p , \ldots, p) \in M_n( \A ) $, and also $ \lambda / n $ is the trace of an approximate splitting of $ p $ into $ n $ Murray-von Neumann equivalent pieces.
\end{proof}

\begin{prop}
Let $ G $ be a countable divisible subgroup of $ \mathbb{R} $ and $ S $ be a countable subset of $ [ 0, 1 ] $.  We can find a separable pseudomatricial C$^*$-algebra $ \A $ so that $ K_0( \A ) \supseteq G \oplus ( \mathbb{Z}^S ) $ as ordered abelian groups when $ G \oplus ( \mathbb{Z}^S ) $ is given the usual lexicographical ordering.
\end{prop}

\begin{proof}
Consider $ \A = \prod_\U M_n $ where $ \U $ is a free ultrafilter on $ \mathbb{N} $.  For $ s \in S $, let $ p_n^{ (s) } $ be a rank $ \lfloor n^{ s } \rfloor $ projection in $ M_n $.  Consider $ P_s = ( p_n^{ (s) } )_\U $, then $ \{ P_s \}_{ s \in S } $ is a countable family of projections in $ \A $.  Note that
$$
\tr( P_s ) = \lim_\U \tr_n( p_n^{ (s) } ) = \lim_\U \frac{ \lfloor n^{ s } \rfloor }{ n } = 0,
$$
so these are trace zero projections.  Also, in $ K_0( \A ) $, $ [ P_s ]_0 \gg [ P_r ]_0 $ when $ s > r $ are in $ S $.

Similarly, for $ \lambda \in G \cap [ 0, 1 ] $ for all $ n $ we can take $ q_n $ to be a rank $ \lfloor n \lambda \rfloor $ projection in $ M_n $.  Then consider $ Q_\lambda = ( q_n )_\U \in \A $.
$$
\tr( Q_\lambda ) = \lim_\U \tr_n ( q_n )_\U = \lim_\U \frac{ \lfloor n \lambda \rfloor }{ n } = \lambda .
$$

We can apply the downward  L\"{o}wenheim-Skolem theorem to the countable set 
$$
\{ P_s , Q_\lambda \}_{ s \in S, \lambda \in G \cap [0,1] }
$$
to get a separable elementary C$^*$-subalgebra $ \A_0 $ of $ \A $ which contains all the $ P_r $ and $ Q_m $.  In particular, this is a pseudomatrical C$^*$-algebra, so the restriction of the trace on $ \A $ is the trace on $ \A_0 $.  Also, for $ n \in \mathbb{N} $, $ r < s $ in $ S $, the sentences $ \varphi_{ n, r, s } $ with parameters from $ \A_0 $  which say ``there are $ n $ orthogonal projections Murray-von Neumann equivalent to $ P_r $ whose sum is dominated by $ P_s $'' hold in $ \A_0 $ for all $ n $ and $ r $.  $ [ P_r ]_0 \ll [ P_s ]_0 $ in $ K_0( \A_0 ) $ as well.  Thus we have shown the range of the trace on projections contains all of $ G \cap [ 0 , 1 ] $.
\end{proof}

\begin{question*}
Can the trace ideal be isomorphic to the compacts?  Can $ \ker( K_0( \tr ) ) \cong \mathbb{Z} $?  Can $ G $ just be $ \mathbb{Q} $?
\end{question*}

\begin{question*}
Is the trace ideal maximal?  What is the quotient of a pseudomatricial C$^*$-algebra by the trace ideal?
\end{question*}

We would be remiss if we did not  mention \cite{MR0006137}, which anticipated many of the ideas behind continuous logic and pseudomatricial C$^*$-algebras.  A similar result from \cite{MR835808} about highly irreducible matrices implies that there is an $ \varepsilon > 0 $ so that the sentence 
$$
\inf_{ || a || \le 1 } \quad \sup_{ p \text{ non-trivial proj.} } || \, a p - p a \, || 
$$
evaluates to at least $ \varepsilon $ in every pseudomatricial C$^*$-algebra except $ M_1 \cong \mathbb{C} $.

\end{section}

\begin{section}{Open Questions}

We conclude with a list of open questions about the pseudocompact and pseudomatricial C$^*$-algebras.

\begin{question*}
Can we find an explicit example of a separable, infinite-dimensional pseudomatricial C$^*$-algebra?
\end{question*}

\begin{question*}
Can we find an explicit axiomatization of the pseudocompact and pseudomatricial C$^*$-algebras?
\end{question*}

\begin{question*}
Can pseudomatricial C$^*$-algebras be exact?  Can they be quasidiagonal?
\end{question*}

\begin{question*}
Are pseudocompact C$^*$-algebras closed under (say minimal) tensor products?
\end{question*}

\end{section}

\bibliographystyle{alpha}
\bibliography{biblio}{}

\end{document}